\documentclass[10pt]{amsart}
\usepackage{subfigure}
\usepackage{graphicx}
\usepackage{enumerate}
\usepackage{bbm}
\usepackage[%breaklinks=true, 
colorlinks=true, urlcolor=blue]{hyperref}

\newtheorem{thm}{Theorem}[section]
\newtheorem{thm*}{Theorem}
\newtheorem{lem}[thm]{Lemma}
\newtheorem{cor}[thm]{Corollary}
\newtheorem{cor*}[thm*]{Corollary}

\newtheorem{prop}[thm]{Proposition}

\theoremstyle{remark}
\newtheorem{remark}[thm]{Remark}

\theoremstyle{definition}

\newcommand{\N}{\mathbbm{N}}
\newcommand{\R}{\mathbbm{R}}

\newcommand{\rd}{\mathrm{d}}

\newcommand{\DN}{\Lambda}

\usepackage{color}

\date{\today}

\title[Boundary determination for a fractional Calder\'{o}n problem]{Boundary determination from local boundary data for a fractional Calder\'{o}n problem}
\author{Heiko Gimperlein and Magnus Goffeng}
\address{Heiko Gimperlein\newline
\indent Engineering Mathematics\newline
\indent Universität Innsbruck\newline
\indent Technikerstraße 13 \newline
\indent 6020 Innsbruck\newline
\indent Austria \newline\newline
\indent Magnus Goffeng,\newline
\indent Centre for Mathematical Sciences\newline 
\indent Lund University \newline 
\indent Box 118, SE-221 00 Lund\newline 
\indent Sweden
}
\email{heiko.gimperlein@uibk.ac.at, magnus.goffeng@math.lth.se}
\begin{document}
\maketitle

\vspace*{-.5cm}

\begin{abstract}
We introduce a new Calder\'{o}n-type problem for fractional powers of Schr\"{o}dinger operators, with local boundary conditions.
The associated Dirichlet-to-Neumann operator maps Dirichlet data to Neumann data on the boundary. We show that this operator determines the Taylor series of the potential at the boundary. In particular, analytic potentials are uniquely determined. These are the first results for a fractional Calder\'{o}n problem with sources and measurements on the boundary.
Our proof builds on recent advances for pseudodifferential boundary value problems to compute the complete symbol of the Dirichlet-to-Neumann operator. 
\\
\end{abstract}

The inverse conductivity problem posed by Calder\'{o}n \cite{cal80} asks whether one can determine the electrical conductivity of a medium in a smooth domain from measurements of electrical currents, when voltages are imposed on the boundary. A closely related problem concerns the determination of the potential in a Schr\"{o}dinger equation from boundary measurements. Assume that there exists a unique solution $u \in H^1(\Omega)$ to the Dirichlet problem 
$$\Delta u + V u = 0 \qquad \text{in } \Omega,$$ 
in a smooth domain $\Omega \subset \mathbbm{R}^n$ for every Dirichlet boundary datum $\gamma_0(u) := u|_{\partial \Omega} = f \in H^{1/2}(\partial \Omega)$. The boundary measurements are encoded in the Dirichlet-to-Neumann operator $\DN_V$ on the boundary $\partial \Omega$, defined as the Neumann boundary data $\DN_V f = \gamma_1(u) := (\nu \cdot \nabla u )|_{\partial \Omega}$ of the solution, where $\nu$ is the outward unit normal to $\partial \Omega$. The Calder\'{o}n problem for the Schr\"{o}dinger equation asks whether one can determine $V$ from $\DN_V$.

A fundamental uniqueness result in dimension $n \geq 3$ was obtained in \cite{syluhl87}: $\DN_V$ determines $V \in C^\infty(\overline{\Omega})$ uniquely. We refer to \cite{fsu,uhl} for overviews of the classical Calder\'{o}n problem and its variants.

In recent years, inverse problems for fractional Laplacians and related operators have attracted considerable interest \cite{feikruuhl, sal17}. In particular, \cite{gsu20} initiated the study of an inverse problem for the fractional Schr\"{o}dinger equation, and quantitative uniqueness results were first obtained in \cite{rulsal20}. One difficulty with these operators is that they are nonlocal, unlike
differential operators. This poses a problem when studying boundary value problems in a domain. Unlike for the classical Calder\'{o}n problem, previous works define a Dirichlet-to-Neumann operator using exterior data, in the complement $\mathbbm{R}^n \setminus \overline{\Omega}$ of the domain, rather than data on the boundary $\partial \Omega$. The work closest in spirit to the problem considered here is \cite{ghosh}, where boundary measurements for exterior sources are considered.\\

This article presents the first results on inverse problems involving fractional operators with \emph{local} boundary conditions, corresponding to sources and measurements on $\partial \Omega$ analogous to those in the classical Calder\'{o}n problem. Our main result is a boundary determination theorem stating that the associated Dirichlet-to-Neumann operator $\DN_{\alpha,V}$ determines all normal derivatives of $V$ at the boundary. Therefore, analytic potentials are uniquely determined in $\Omega$. Since $\DN_{\alpha,V} = \DN_V$ when $\alpha=1$, this embeds the classical inverse problem in an analytic one-parameter family in $\alpha$.\\

More precisely, for $\alpha\in (0,1)$, the fractional Laplacian $\Delta^\alpha$ on $\mathbbm{R}^n$ is defined as a pseudodifferential operator of order $2 \alpha \in (0,2)$ or, equivalently, as a singular integral operator,
\begin{equation}
\Delta^\alpha u(x) = \mathcal{F}_{\xi \mapsto x}\left( |\xi|^{2\alpha} \hat{u}(\xi)\right) = c_{n,\alpha}\ \mathrm{pv.} \int_{\mathbbm{R}^n}\frac{u(x)-u(x+y)}{|y|^{n+2\alpha}} dy.
\end{equation}
Our goal is to study this and related operators acting on functions supported in a compact, connected and smoothly bounded domain $\Omega \subset \mathbbm{R}^n$, 
or more generally in $\Omega \subset M $. Here, $(M ,g)$ is a complete Riemannian manifold of dimension $n\geq 2$ with bounded geometry. Following Seeley \cite{seeley67}, we may define the fractional powers $\Delta_V^\alpha := (\Delta+V)^\alpha$, for a given potential $V \in C^\infty_c(M )$ and $\Delta = \Delta_{{g}}$ the Laplace-Beltrami operator.

For the operator $\Delta_V^\alpha$, we consider the nonhomogeneous Dirichlet problem
\begin{align}\label{fracprob}
\begin{split}
r^+\Delta_V^\alpha u & = 0 \mbox{ in } \Omega,\\
\gamma_0(d^{1-\alpha} u) & = f \mbox{ on } \partial \Omega,
\end{split}
\end{align}
where $r^+\Delta_V^\alpha u := (\Delta_V^\alpha u)|_\Omega$ and $d$ is a smooth extension of the distance to $\partial \Omega$. The problem \eqref{fracprob} is well posed in appropriate transmission spaces by \cite{grubb18}. A Dirichlet-to-Neumann map associated to \eqref{fracprob} was introduced in \cite{grubb18} as
\begin{align}
\label{dnalpdef}
\DN_{\alpha, V}(f) := \alpha \gamma_1(d^{1-\alpha} u).
\end{align}
As in the case of the ordinary Laplacian, $\DN_{\alpha,V}$ is an elliptic, classical pseudodifferential operator on $\partial \Omega$ of order $1$ \cite[Theorem 6.1]{grubb18}. We remark that the pre-factor $\alpha$ in \eqref{dnalpdef} arises from normalisations of fractional trace operators appearing in \cite{grubb18}, see the discussion after Equation \eqref{fracprob2} below. The main result of this work is: 

\begin{thm*}
\label{aljnkajdna}
For $\alpha\in (0,1]$ and any open $U\subseteq \partial \Omega$, the first $k+2$ terms of the complete symbol of $\DN_{\alpha,V}$ on $U$ determine $\alpha$ and the $k$ terms $\partial_{\nu}^{j} V|_{U}$ for $0\leq j\leq k-1$.
\end{thm*}

Theorem \ref{aljnkajdna} is a consequence of Theorem \ref{alknalsdnaljkn} and Corollary \ref{adaldknal}. In the case of analytic potentials, this boundary determination result also implies a solution to the Calder\'{o}n problem:

\begin{cor*}
Assume that $M $ and $V$ are analytic in a neighborhood of $\overline{\Omega}$. Then $\DN_{\alpha,V}$ uniquely determines $\alpha\in (0,1]$ and $V$ in $\Omega$.
\end{cor*}

Our proof applies to slightly more general Laplace-type differential operators $A$ than $A=\Delta+V$, as in \eqref{eq:Aoperator} below. It builds on recent advances by Grubb for fractional boundary value problems using pseudodifferential methods, beginning with \cite{grubb15}, which make it possible to study $\DN_{\alpha,V}$ \cite{grubb18}. We obtain its complete symbol and recursively determine all normal derivatives of $V$, leading to the main technical result Theorem \ref{alknalsdnaljkn}. This approach mirrors the work by Kohn-Vogelius \cite{kohnvog84}, further developed by Lee-Uhlmann and Sylvester-Uhlmann in \cite{leeuhl89,syluhl88}, which used pseudodifferential methods to show the boundary determination of the conductivity (or metric): They proved, similarly, that the complete symbol of the Dirichlet-to-Neumann operator determines all normal derivatives of the conductivity. Stronger results for problems related to the current article, including boundary determination of $g$ and stability estimates, may follow by constructing CGO-like solutions; this is pursued in \cite{GGR}.\\

\noindent \emph{Notation:} We follow the standard notation for pseudodifferential operators \cite{GrubbGreenBook}. The Fourier transform on $\R^n$ is denoted by $\mathcal{F}$ and is given by
$$\mathcal{F}f(\xi):=\int_{\R^n}f(x)\,\mathrm{e}^{-ix\cdot\xi}\rd x.$$
We set $D_x=-i\frac{\partial}{\partial x}$. For $\beta=(\beta_1,\ldots,\beta_n)\in \N^n$, we write $|\beta|=\sum_j \beta_j$, $D^\beta_x=D_{x_1}^{\beta_1}\cdots D_{x_n}^{\beta_n}$ and $x^\beta=x_1^{\beta_1}\cdots x_n^{\beta_n}$. In this convention, for a Schwartz function $f$ on $\R^n$,
$$\mathcal{F}(D_x^\beta f)(\xi)=\xi^\beta\mathcal{F}f(\xi)\quad\mbox{and}\quad \mathcal{F}(x^\beta f)(\xi)=(-D_\xi)^\beta\mathcal{F}f(\xi),$$
and the composition of two pseudodifferential operators with symbols $p,q$ is, up to smoothing operators, an operator with symbol 
$$p*q\sim\sum_\beta \frac{1}{\beta !} \partial_\xi^\beta pD_x^\beta q=\sum_\beta \frac{1}{\beta !} D_\xi^\beta p\partial_x^\beta q.$$

\noindent \emph{Acknowledgments:} We thank Angkana R\"{u}land for pointing out this problem to us and Gerd Grubb, Rafe Mazzeo, Angkana R\"{u}land and Mikael Persson Sundqvist for fruitful discussions. MG was supported by the Swedish Research Council Grant VR 2025-03923. HG would like to thank the Isaac Newton Institute for Mathematical Sciences, Cambridge, for support and hospitality during the programme \emph{Geometric Spectral Theory and Applications} where work on this paper was undertaken. This work was supported by EPSRC grant no EP/Z000580/1.\\

\noindent \emph{AI usage:} The authors acknowledge the usage of AI. It was used for proofreading the paper, and discussions informed Subsection \ref{aonaldna}.

\section{Set-up and outline}
\label{sec:setup}

In this section we give an outline of how pseudodifferential methods allow us to determine the full symbol of the Dirichlet-to-Neumann operator and study its dependence on $V$. Classical references on pseudodifferential boundary problems include \cite{eskinbook,hornotes}. The approach has more recently been developed especially by Grubb, beginning with \cite{grubb14,grubb15}, and \cite{gimpgofflouc,gimpgofflouctech,grubb16,grubb18,grubb20} are particularly relevant for the present work.

Recall from the introduction that $(M ,g)$ is a complete Riemannian manifold of dimension $n\geq 2$. On $M $ we consider a second-order differential operator of Laplace-type, 
\begin{align}\label{eq:Aoperator}
A&= \sum_{|\beta|=2} g^{\beta}(x) D^\beta_x +\sum_{|\beta|=1} b_\beta(x) D_x^\beta + a_0(x)\\& = a_2(x,D)+a_1(x,D)+a_0(x), \nonumber
\end{align}
i.e.~an operator with the same principal symbol $a_2$ as the Laplace-Beltrami operator and smooth, compactly supported lower-order terms $a_0$, $a_1$.  We sometimes write $g_x(\xi,\xi)=|\xi|_x^2$ or $g_x(\xi,\xi)=|\xi|^2$ when the $x$-dependence is clear from context.

For this exposition, we assume that either $M $ is compact, or $M = \mathbbm{R}^n$ and that $A$ coincides with the Euclidean Laplacian outside a compact set. Following Seeley \cite{seeley67}, see also \cite{GrubbGreenBook}, the fractional powers of the sectorial operator $A$ on $M $, given by the contour integral 
\begin{equation}
\label{lkjnljknadljadn}
P=A^\alpha=-\frac{1}{2\pi i}\int_\Gamma \lambda^{\alpha}(A-\lambda)^{-1} \rd \lambda
\end{equation}
along an appropriate contour $\Gamma \subset \mathbb{C}$, define a pseudodifferential operator $P$ of order $2\alpha$ on $M $.

We use the pseudodifferential approach in \cite{grubb18} to solve a nonhomogeneous boundary problem for $P$ in a smooth, bounded domain $\Omega \subset M $, 
\begin{align}
\label{fracprob2}
\begin{split}
r^+P u & = 0 \mbox{ in } \Omega,\\
\gamma_0^{\alpha-1}(u) & = f \mbox{ on } \partial \Omega,
\end{split}
\end{align} 
prescribing the Dirichlet trace $\gamma_0^{\alpha-1}(u) := \Gamma(\alpha) \gamma_0(d^{1-\alpha} u)$. To do so, one computes a factorization $P \sim P^- P^+$ in ``minus'' and ``plus'' operators, which preserve support in $M \setminus \Omega$, respectively $\overline{\Omega}$. An explicit formula for the full symbol $s$ of the operator 
$$\DN_\alpha(f) = \gamma_1^{\alpha-1}(u) := \Gamma(\alpha+1) \gamma_1(d^{1-\alpha} u)$$ 
is then obtained in local coordinates $x=(x', x_n) \in \mathbbm{R}^n_+$ from \cite[Theorem 6.1]{grubb18}:
\begin{equation}
\label{adlknadlkna}
s(x',\xi')=-\lim_{z_n\to 0^+}\mathcal{F}^{-1}_{\xi_n\to z_n} q^+(x',0,\xi',\xi_n)-\alpha| \xi'|.
\end{equation}
Here, $q^+$ is the full symbol of $Q^+= P^+(\Xi_+^\alpha)^{-1}$, a reduction of $P^+$ to order $0$ using an appropriate order-reducing operator $\Xi_+^\alpha$. We note that \cite[Theorem 6.1]{grubb18} uses $\langle \xi'\rangle$ rather than $|\xi'|$, but for the symbol computation of $\DN_\alpha$ this has no impact as we use also $|\xi'|$ in the computation of $q^+$. We also note that the formulas of \cite[Theorem 6.1]{grubb18} involve the projection $h^+$ (cf. \cite[Appendix]{grubb18}), which for the rational functions we consider plays no role after having taken the limit $z_n\to 0^+$. Our main technical results clarify how this symbol $s$ depends on the symbol of $A$ from \eqref{eq:Aoperator}. Following the above approach, the proof is divided into $5$ steps as follows. \\

\noindent \textbf{Step 1: Resolvent (Subsection \ref{subsec21}).} The resolvent $(A-\lambda)^{-1}$ in the contour integral \eqref{lkjnljknadljadn} is a 
pseudodifferential operator with parameter $\lambda$ of order $-2$. Treating $\lambda$ as another covariable, the parameter-dependent pseudodifferential calculus \cite{GrubbGreenBook,shubinbook} allows us to construct the complete symbol $b$ of the resolvent $(A-\lambda)^{-1}$,
$$b\sim \sum_{j=0}^\infty b_{-2-j}(x,\xi,\lambda).$$
Each term $b_{-2-j}$ depends smoothly on $x$ and $(\xi,\lambda)\neq 0$ and is homogeneous of order $-2-j$ in $(\xi,\lambda)$ under the scaling $t\cdot(\xi,\lambda):=(t\xi,t^2\lambda)$. The symbols $b_{-2-j}$ are defined inductively from the symbol of $A$, i.e.~from $a_0$, $a_1$, $g$ and their derivatives. \\

\noindent \textbf{Step 2: Fractional powers (Subsection \ref{subsec22}).} The contour integral \eqref{lkjnljknadljadn} relates the complete symbol $p$ of $P=A^\alpha$ to the symbol $b$ of the resolvent $(A-\lambda)^{-1}$,
$$p\sim \sum_{j=0}^\infty p_{2\alpha-j}(x,\xi).$$
Here, each $p_{2\alpha-j}$ is a smooth function of $x$ and $\xi\neq 0$, which is homogeneous of order $2\alpha-j$ in $\xi$ and computable from $b_{-2-j}(x,\xi,\lambda)$. The resulting explicit expressions show that $p_{2\alpha-j}/p_{2\alpha}$ restricts to a polynomial on the cosphere $|\xi|=1$, which depends on $\alpha$ in a polynomial way.\\

\noindent \textbf{Step 3: Wiener-Hopf factorization (Subsection \ref{subsec23}).} Near the boundary $\partial \Omega$, we can in local coordinates $x=(x', x_n) \in \mathbbm{R}^n_+$ determine the factorization of the symbol $p$,
$$p\sim p^-*p^+\equiv \sum_{\beta\in \mathbbm{N}^n}\frac{1}{\beta!} \partial_\xi^\beta p^-D_x^\beta p^+.$$
Here, $p^\pm\sim \sum_{j=0}^\infty p^\pm_{\alpha-j}(x,\xi)$ extend holomorphically to the complex half-plane $\mp \mathrm{Im}(\xi_n)>0$, and the principal symbols $p^\pm_{\alpha}(x,\xi)$ are invertible. The leading term is explicitly factorized as $p_{2\alpha}=p^-_\alpha p^+_\alpha$, and $p^\pm_{\alpha-j}$ are then determined inductively,
\begin{equation}
\label{pfdecom}
\frac{p_{2\alpha-j}}{p_{2\alpha}}-\frac{1}{p_{2\alpha}}\sum_{\substack{k+l+|\beta|=j\\k,l<j}}\frac{1}{\beta!}\partial_\xi^\beta p^-_{\alpha-k}D_x^\beta p^+_{\alpha-l}=\mathfrak{q}^+_j+\mathfrak{q}^-_j.
\end{equation}
Here, $\mathfrak{q}^{\pm}_j$ are rational functions in $\xi_n$, whose coefficients are homogeneous in $\xi'$. They are determined by holomorphicity in $\mp \mathrm{Im}(\xi_n)>0$. Then 
$p^\pm_{\alpha-j}:=\mathfrak{q}^{\pm}_jp^\pm_{\alpha}$.
We refer to \cite{grubb14, grubb18,hornotes} for the analytical properties of $p^\pm$ in the case of operators like the fractional Laplacian and to \cite{gimpgofflouctech} for related explicit computations. \\

\noindent \textbf{Step 4: Symbol of $\DN_\alpha$ (Subsection \ref{subsec24}).} The symbol $s\sim \sum_j s_{1-j}(x',\xi')$ is explicitly computed using \eqref{adlknadlkna} as 
$$s_{1-j}(x',\xi')=-\lim_{z_n\to 0^+}\mathcal{F}^{-1}_{\xi_n\to z_n} \mathfrak{q}^+_j(x',\xi',\xi_n), \qquad j>0,$$
using the residue theorem to compute the inverse Fourier transform of the rational function $\mathfrak{q}^+_j$. As a result,
\begin{equation}\label{sformula}
s_{1-j}(x',\xi')=c_j(\alpha)\partial_{x_n}^{j-2}a_0(x')|\xi'|_x^{1-j}+\tilde{a}_{1-j}(x',\xi'),
\end{equation} 
for a universal polynomial $c_j(\alpha)$ of degree $j-1$ and a symbol $\tilde{a}_{1-j}$ which only depends on derivatives of $a_0$ of order $\leq j-3$ and tangential derivatives of order $\leq j-2$. \\

\noindent \textbf{Step 5: Boundary determination of $V=a_0$ (Section \ref{kjnjnad}).} Equation \eqref{sformula} determines the first $k$ terms in the Taylor expansion of $a_0$ at the boundary from the first $k+2$ terms $s_1,\ldots, s_{-k}$ in the symbol of $\DN_{\alpha}$, as long as $c_j(\alpha)\neq 0$. As the polynomial $c_j(\alpha)$ is universal, we can perform the computation in a special case. By choosing a particular potential in the half-space $\R^n_+$, we show in Section \ref{kjnjnad} that $c_j(\alpha)<0$ for all $j>1$ and $\alpha>0$. 

Moreover, %in Section \ref{kjnjnad} 
we consider the model case $A=\Delta+V$ in the half space $\R^n_+$, where we in Equation \eqref{symbolcomad} compute that
\begin{align*}
s_1(x',\xi')=&-\alpha |\xi'|,\\
s_0(x',\xi')=&\ 0,\\
s_{-1}(x',\xi')=&-\frac{\alpha}{2} a_0(x') |\xi'|^{-1}, \qquad c_2(\alpha)=-\frac{\alpha}{2},\\
s_{-2}(x',\xi')=&-\frac{\alpha^2}{4|\xi'|^2}\partial_{x_n} a_0 -\frac{i\alpha}{4|\xi'|^3} L' a_0, \quad c_3(\alpha)=-\frac{\alpha^2}{4},
\end{align*}
with $L'=\xi'\cdot\nabla_{x'}$ tangential to the boundary.

\begin{remark}
It is a natural question to study the Calder\'{o}n problem also for operators of the form $P=\Delta_g^\alpha+V$, as in \cite{feikruuhl, rulsal20, sal17}. Such operators $P$ are not classical pseudodifferential operators, and their analysis would require a nontrivial extension of the approach in the current article. Further discussion of such operators can be found in \cite[Section 4]{grubb19}.
\end{remark}

\section{The symbol of $\DN_\alpha$ and its structure}
\label{strcutlnjn}

In this section we follow the steps of the proof outlined in Section \ref{sec:setup}.

\subsection{Step 1: Resolvent}
\label{subsec21}
The resolvent $(A-\lambda)^{-1}$ is as discussed above an elliptic pseudodifferential operator with parameter of order $-2$, see \cite{GrubbGreenBook,shubinbook}. To shorten notation, we denote the symbol of $A$ by $a$ and recall that $a_2$ is given by the metric $g$. One finds the symbol with parameter $b$ of the resolvent from the ansatz
$$b\sim \sum_{j=0}^\infty b_{-2-j}(x,\xi,\lambda),$$
for which the equation $(a-\lambda)*b\sim 1$ determines the symbols $b_{-2-j}$ inductively. Indeed, the leading term in the equation gives 
$$b_{-2}(x,\xi,\lambda)=(g_x(\xi,\xi)-\lambda)^{-1},$$
and for $j>0$, 
\begin{align}
\nonumber
b_{-2-j}(x,\xi,\lambda)=&-(g_x(\xi,\xi)-\lambda)^{-1}\sum_{\substack{k+l+|\beta|=j\\l<j}}\frac{1}{\beta!} \partial_\xi^\beta a_{2-k}(x,\xi)D_x^\beta b_{-2-l}(x,\xi,\lambda)=\\
\label{formualforb}
=&-(g_x(\xi,\xi)-\lambda)^{-1}\sum_{\substack{l+|\beta|=j\\l<j}}\frac{1}{\beta!} \partial_\xi^\beta a_{2}(x,\xi)D_x^\beta b_{-2-l}(x,\xi,\lambda)-\\
\nonumber
&-(g_x(\xi,\xi)-\lambda)^{-1}\sum_{\substack{l+|\beta|=j-1\\l<j}}\frac{1}{\beta!} \partial_\xi^\beta a_{1}(x,\xi)D_x^\beta b_{-2-l}(x,\xi,\lambda)-\\
\nonumber
&-(g_x(\xi,\xi)-\lambda)^{-1}a_{0}(x)b_{-j}(x,\xi,\lambda).
\end{align}

We can compute from the formulas above that 
\begin{align*}
b_{-3}(x,\xi,\lambda)=&a_1(a_2-\lambda)^{-2}+\sum_{|\beta|=1} \partial_\xi^\beta a_2 D_x^\beta a_2 (a_2-\lambda)^{-3}=\\
=&a_1(x,\xi)(|\xi|_x^2-\lambda)^{-2}+g_x(\xi,D_x |\xi|_x^2) (|\xi|_x^2-\lambda)^{-3}
\end{align*}
and more generally an induction argument combined with Equation \eqref{formualforb} gives the following structural statement.

\begin{prop}
\label{synmbkokad}
For any $j$ there is a collection of homogeneous polynomial symbols $\hat{p}_{l,j}$ of order $2l+2$ for even $j$ and order $2l+1$ for odd $j$ such that 
$$b_{-2-j}(x,\xi,\lambda)=\sum_{l=0}^{ j} \hat{p}_{l,j}(x,\xi)(|\xi|_x^2-\lambda)^{-\lfloor j/2\rfloor-l-2}$$
Each polynomial $\hat{p}_{l,j}$ is given by a universal polynomial evaluated on the data 
\begin{align*}
\bigg(\xi, |\xi|_x^2, D_x |\xi|_x^2, &\ldots, D_x^j |\xi|_x^2,\\
&a_1(x,\xi), D_x a_1(x,\xi), \ldots, D_x^{j-1}a_1(x,\xi), \\
&\qquad\qquad\qquad a_0(x), D_x a_0(x), \ldots, D_x^{j-2}a_0(x)\bigg),
\end{align*}
by contracting with the metric tensor.
\end{prop}

\subsection{Step 2: Fractional powers}
\label{subsec22}
We now compute the complete symbol $p$ of the fractional power $A^\alpha$. It takes the form
$$p\sim \sum_{j=0}^\infty p_{2\alpha-j}(x,\xi),$$
where each $p_{2\alpha-j}$ is explicitly computed from residue calculus as a contour integral
$$p_{2\alpha-j}(x,\xi)=-\frac{1}{2\pi i}\int_\Gamma \lambda^{\alpha} b_{-2-j}(x,\xi,\lambda)\rd \lambda,$$
and $\Gamma$ is an appropriate contour. For detailed expositions, see \cite{GrubbGreenBook,seeley67,shubinbook}.

We have $p_{2\alpha}(x,\xi)=|\xi|_x^{2\alpha}$. Performing the contour integral, we compute that 
\begin{align*}
p_{2\alpha-1}(x,\xi)=&\alpha a_1(x,\xi)|\xi|_x^{2\alpha-2}+\frac{\alpha(\alpha-1)}{2}|\xi|_x^{2\alpha-4}\sum_{|\beta|=1} \partial_\xi^\beta a_2D_x^\beta a_2= \\
=&\alpha a_1(x,\xi)|\xi|_x^{2\alpha-2}+\frac{\alpha(\alpha-1)}{2}g_x(\xi,D_x |\xi|_x^2) |\xi|_x^{2\alpha-4}.
\end{align*}
More generally we can perform the contour integral on the structural expressions in Proposition \ref{synmbkokad} and obtain the following. We recall the notation for the binomial coefficient ``$x$ choose $m$'' for a real number $x$ and a natural number $m$:
$$\begin{pmatrix} x\\ m\end{pmatrix}:=\frac{x(x-1)\cdots (x-m+1)}{m!}.$$

\begin{prop}
\label{alnlknjnad}
The collection of homogeneous polynomial symbols $\hat{p}_{l,j}$ from Proposition \ref{synmbkokad} computes the complete symbol expansion of $A^\alpha$ by
$$p_{2\alpha-j}(x,\xi)=\sum_{l=0}^{j} \begin{pmatrix}\alpha\\ \lfloor j/2\rfloor+l+1\end{pmatrix} \hat{p}_{l,j}(x,\xi)|\xi|_x^{2\alpha-2\lfloor j/2\rfloor-2l-2}.$$
\end{prop}

For later purposes, we note that for order reasons we have the following. 

\begin{lem}
\label{complapda}
The collection of homogeneous polynomial symbols $\hat{p}_{l,j}$ from Proposition \ref{synmbkokad} are of the form 
$$\hat{p}_{l,j}(x,\xi)=\sum_{k=0}^{\lfloor l/2\rfloor}c_{j,l,k} \partial_{x_n}^{j-2}a_0(x)\xi_n^{l-2k}|\xi|_x^{2k}+\tilde{p}_{l,j}(x,\xi).$$
Here, $c_{j,l,k}$ are universal rational numbers, and each polynomial $\tilde{p}_{l,j}$ is given by a universal polynomial evaluated on the data 
\begin{align*}
\bigg(\xi, |\xi|_x^2, D_x |\xi|_x^2, &\ldots, D_x^j |\xi|_x^2,\\
&a_1(x,\xi), D_x a_1(x,\xi), \ldots, D_x^{j-1}a_1(x,\xi), \\
&\qquad\qquad a_0(x), D_x a_0(x), \ldots, D_x^{j-3}a_0(x),D_{x'}^{j-2}a_0(x)\bigg),
\end{align*}
by contracting with the metric tensor.
\end{lem}

\subsection{Step 3: Wiener-Hopf factorization}
\label{subsec23}
Next, we consider the pièce de résis\-tance in computing the symbol of the Dirichlet-to-Neumann operator: the Wiener-Hopf factorization of $p$. We are looking for a factorization 
$$p\sim p^-*p^+,$$
where $p^\pm\sim \sum_{j=0}^\infty p^\pm_{\alpha-j}(x,\xi)$ extends holomorphically to $\mp \mathrm{Im}(\xi_n)>0$ and $p^\pm_{\alpha}(x,\xi)$ is invertible. Existence of such factorizations was proven in \cite{grubb14,hornotes}. The principal symbol factors as $p_{2\alpha}= p_{\alpha}^-p_{\alpha}^+$, where 
$$p_{\alpha}^\pm=(\xi_n-h_\pm)^\alpha.$$
The lower order terms $p^\pm_{\alpha-j}$ are computed inductively from writing out the equation $p\sim p^-*p^+$. We write the terms in the form $p^\pm_{\alpha-j}=\mathfrak{q}^{\pm}_jp^\pm_{\alpha}$ and find $\mathfrak{q}^{\pm}_j$ from $p^\pm_{\alpha},p^\pm_{\alpha-1}\ldots, p^\pm_{\alpha-j+1}$ by decomposing
\begin{equation}
\label{pfdecom2}
\frac{p_{2\alpha-j}}{p_{2\alpha}}-\frac{1}{p_{2\alpha}}\sum_{\substack{k+l+|\beta|=j\\k,l<j}}\frac{1}{\beta!}\partial_\xi^\beta p^-_{\alpha-k}D_x^\beta p^+_{\alpha-l}=\mathfrak{q}^+_j+\mathfrak{q}^-_j,
\end{equation}
where $\mathfrak{q}^{\pm}_j$ are the rational functions in $\xi_n$ with coefficients in homogeneous functions of $\xi'$ that are determined from the property of having no poles $\mp \mathrm{Im}(\xi_n)>0$. A key difference between our problem and the general results of \cite{grubb14,hornotes} is that the symbol $p$ that we factorize satisfies that $\frac{p_{2\alpha-j}}{p_{2\alpha}}$ is a rational function for any $j$, so the decomposition \eqref{pfdecom2} can be computed from partial fraction decompositions. This methodology was heavily used in the explicit computations of \cite{gimpgofflouctech}. The reader is also encouraged to consult the appendices of the unpublished, expanded version \cite{gimpgofflouctechv1}, where a plethora of computational methods can be found. In particular, the appendices of \cite{gimpgofflouctechv1} together with Lemma \ref{complapda} imply the following structural description of $p_{\alpha-j}^\pm$.

\begin{lem}
\label{complapdadadida}
The collection of symbols $p_{\alpha-j}^\pm$ can be written as 
$$p_{\alpha-j}^\pm=\check{p}^\pm_{\alpha-j}+\tilde{p}_{\alpha-j}^\pm,$$
where 
$$\check{p}^\pm_{\alpha-j}(x,\xi',\xi_n)=\partial_{x_n}^{j-2}a_0(x)\sum_{k+l-m=-j} c_{k,l,m}(\alpha)\xi_n^k|\xi'|^l(\xi_n\mp i|\xi'|)^{\alpha-m}$$
for universal, rational polynomials $c_{k,l,m}(\alpha)$ in $\alpha$ and $\tilde{p}_{\alpha-j}^\pm$ is given as a sum of terms of the form $\xi_n^k|\xi'|^l(\xi_n\mp i|\xi'|)^{\alpha-m}$ multiplied by a universal polynomial evaluated on the data 
\begin{align*}
\bigg(\xi, |\xi|_x^2, D_x |\xi|_x^2, &\ldots, D_x^j |\xi|_x^2,\\
&a_1(x,\xi), D_x a_1(x,\xi), \ldots, D_x^{j-1}a_1(x,\xi), \\
&\qquad\qquad a_0(x), D_x a_0(x), \ldots, D_x^{j-3}a_0(x),D_{x'}^{j-2}a_0(x)\bigg),
\end{align*}
by contracting with the metric tensor. 
\end{lem}

\begin{remark}
The collection of the symbols $\check{p}_{\alpha-j}^+$ is inductively determined by decomposing 
$$\frac{\check{p}_{2\alpha-j}}{p_{2\alpha}}-\frac{1}{p^+_{\alpha}}\sum_{q=1}^j\begin{pmatrix}\alpha\\ q\end{pmatrix} (\xi_n+i|\xi'|)^q D_{x_n}^q \check{p}^+_{\alpha-j+q}=\check{\mathfrak{q}}^+_j+\check{\mathfrak{q}}^-_j,$$
where $\check{p}_{2\alpha-j}$ contains the leading $x_n$-derivatives as computed in Lemma \ref{complapda}:
$$\check{p}_{2\alpha-j}(x,\xi)=\sum_{l=0}^{ j} \sum_{k=0}^{\lfloor l/2\rfloor} \begin{pmatrix}\alpha\\ \lfloor j/2\rfloor+l+1\end{pmatrix}c_{j,l,k} \partial_{x_n}^{j-2}a_0(x)\xi_n^{l-2k}|\xi|_x^{2k}|\xi|_x^{2\alpha-2\lfloor j/2\rfloor-2l-2}.$$
The functions $\mathfrak{q}^{\pm}_j$ are rational in $\xi_n$ with coefficients in homogeneous functions of $\xi'$ and are determined by their holomorphicity in $\mp \mathrm{Im}(\xi_n)>0$. We then set 
$$\check{p}^+_{\alpha-j}:=\check{\mathfrak{q}}^{+}_jp^+_{\alpha}.$$
\end{remark}

In the next subsection, we proceed with our structural discussion. First, we end this subsection with an example of the subleading term in the Wiener-Hopf factorization of $p$. We need to decompose the following expression according to holomorphicity in the upper/lower half-plane:
\begin{align*}
\frac{p_{2\alpha-1}}{p_{2\alpha}}&-\frac{1}{p_{2\alpha}}\sum_{|\beta|=1} \partial_\xi^\beta p_{\alpha}^-D_x^\beta p_{\alpha}^+=\\
=&\alpha\frac{a_1(x,\xi)}{|\xi|_x^2}+\frac{\alpha(\alpha-1)}{2}\frac{\sum_{|\beta|=1} \partial_\xi^\beta a_2D_x^\beta a_2}{|\xi|_x^{4}}-\frac{\sum_{|\beta|=1} \partial_\xi^\beta p_\alpha^-D_x^\beta p_\alpha^+}{|\xi|_x^{2\alpha}}=\\
=&\alpha\frac{a_1(x,\xi)}{(\xi_n-h_+)(\xi_n-h_-)}+\\
&+\frac{\alpha(\alpha-1)}{2}\frac{\sum_{|\beta|=1} \partial_\xi^\beta [(\xi_n-h_+)(\xi_n-h_-)]D_x^\beta [(\xi_n-h_+)(\xi_n-h_-)]}{(\xi_n-h_+)^2(\xi_n-h_-)^2}-\\
&-\alpha^2\frac{\sum_{|\beta|=1} \partial_\xi^\beta (\xi_n-h_-)D_x^\beta (\xi_n-h_+)}{(\xi_n-h_+)(\xi_n-h_-)}.
\end{align*}
Using that 
\begin{align*}
\sum_{|\beta|=1} &\partial_\xi^\beta [(\xi_n-h_+)(\xi_n-h_-)]D_x^\beta [(\xi_n-h_+)(\xi_n-h_-)]=\\
&=i\sum_{\gamma,\epsilon\in \{+,-\}}\left(\partial_{x_n}h_\gamma(\xi_n-h_\gamma)(\xi_n-h_\epsilon)+\nabla_{x'}h_\gamma\cdot\nabla_{\xi'}h_\epsilon(\xi_n-h_\gamma)(\xi_n-h_\epsilon)\right),
\end{align*}
and 
\begin{align*}
\partial_\xi^\beta (\xi_n-h_-)D_x^\beta (\xi_n-h_+)=-i\partial_{x_n}h_++i\nabla_{x'}h_+\cdot\nabla_{\xi'}h_-,
\end{align*}
we obtain the useful expression
\begin{align*}
\frac{p_{2\alpha-1}}{p_{2\alpha}}&-\frac{1}{p_{2\alpha}}\sum_{|\beta|=1} \partial_\xi^\beta p_{\alpha}^-D_x^\beta p_{\alpha}^+=\\
=&\alpha\frac{a_1(x,\xi)}{(\xi_n-h_+)(\xi_n-h_-)}-i\alpha^2\frac{-\partial_{x_n}h_++\nabla_{x'}h_+\cdot\nabla_{\xi'}h_-}{(\xi_n-h_-)(\xi_n-h_+)}+\\
&+\frac{i\alpha(\alpha-1)}{2}\bigg(\frac{\partial_{x_n}h_++\partial_{x_n}h_-+\nabla_{x'}h_-\cdot\nabla_{\xi'}h_++\nabla_{x'}h_+\cdot\nabla_{\xi'}h_-}{(\xi_n-h_-)(\xi_n-h_+)}+\\
&\qquad\qquad\qquad\qquad\qquad+\frac{\partial_{x_n}h_++\nabla_{x'}h_+\cdot\nabla_{\xi'}h_+}{(\xi_n-h_-)^2}+\frac{\partial_{x_n}h_-+\nabla_{x'}h_-\cdot\nabla_{\xi'}h_-}{(\xi_n-h_+)^2}\bigg).
\end{align*}
To decompose this, first write 
$$a_1(x,\xi)=a_{10}(x,\xi')+a_{11}(x)\xi_n.$$
Then we have 
\begin{align*}
\frac{a_1(x,\xi)}{(\xi_n-h_+)(\xi_n-h_-)}=&\frac{a_{10}+a_{11}\xi_n}{(\xi_n-h_+)(\xi_n-h_-)}=\\
=&\left(\frac{a_{11}}{2}\left(1+\frac{h_++h_-}{h_+-h_-}\right)+\frac{a_{10}}{h_+-h_-}\right)\frac{1}{\xi_n-h_+}+\\
&+\left(\frac{a_{11}}{2}\left(1-\frac{h_++h_-}{h_+-h_-}\right)-\frac{a_{10}}{h_+-h_-}\right)\frac{1}{\xi_n-h_-},
\end{align*}
and therefore, 
\begin{align*}
\partial_\xi^\beta (\xi_n-h_-)D_x^\beta (\xi_n-h_+)=-i\partial_{x_n}h_++i\nabla_{x'}h_+\cdot\nabla_{\xi'}h_-.
\end{align*}
Continuing the computation from above, we conclude
\begin{align*}
\frac{p_{2\alpha-1}}{p_{2\alpha}}&-\frac{1}{p_{2\alpha}}\sum_{|\beta|=1} \partial_\xi^\beta p_{\alpha}^-D_x^\beta p_{\alpha}^+=\mathfrak{q}_1^++\mathfrak{q}_1^-,
\end{align*}
where we write 
$$\mathfrak{q}_1^\pm(x,\xi)=\tau_{12}^\pm(x,\xi')(\xi_n-h_\pm(x,\xi'))^{-2}+\tau_{11}^\pm(x,\xi')(\xi_n-h_\pm(x,\xi'))^{-1},$$
and 
\begin{align*}
\tau_{12}^\pm=&\frac{i\alpha(\alpha-1)}{2}(\partial_{x_n}h_\mp+\nabla_{x'}h_\mp\cdot\nabla_{\xi'}h_\mp),\\
\tau_{11}^\pm=&\frac{\alpha a_{11}}{2}\left(1\pm \frac{h_++h_-}{h_+-h_-}\right)\\
&\pm\frac{\alpha a_{10}-\frac{i\alpha(\alpha+1)}{2}\nabla_{x'}h_+\cdot\nabla_{\xi'}h_-+\frac{i\alpha(\alpha-1)}{2}\nabla_{x'}h_-\cdot\nabla_{\xi'}h_++\frac{i\alpha(3\alpha\mp 1)}{2}\partial_{x_n}h_-}{h_+-h_-}.
\end{align*}
In particular, 
$$p_{\alpha-1}^\pm(x,\xi)=\tau_{12}^\pm(x,\xi')(\xi_n-h_\pm(x,\xi'))^{\alpha-2}+\tau_{11}^\pm(x,\xi')(\xi_n-h_\pm(x,\xi'))^{\alpha-1}.$$

\subsection{Step 4: Symbol of $\DN_\alpha$}
\label{subsec24}

We now aim to determine 
$$s_{1-j}(x',\xi')=-\lim_{z_n\to 0^+}\mathcal{F}^{-1}_{\xi_n\to z_n} \mathfrak{q}^+_{j}(x',0,\xi',\xi_n).$$
Let us first carry out this computation for $j=1$. Note that on $x_n=0$ we have $h_\pm(x',0,\xi')=\pm i|\xi|_{x'}$. Therefore, by residue calculus we compute that
\begin{align*}
s_0(x',\xi')=&-\tau_{12}^+(x,\xi')\lim_{z_n\to 0^+}\mathcal{F}^{-1}_{\xi_n\to z_n} [(-i|\xi'| +\xi_n)^{-2}]\\
&-\tau_{11}^+(x,\xi')\lim_{z_n\to 0^+}\mathcal{F}^{-1}_{\xi_n\to z_n} [(-i|\xi'| +\xi_n)^{-1}]=-i\tau_{11}^+(x,\xi').
\end{align*}
More generally, we employ the following lemma that follows from residue calculus.

\begin{lem}
\label{fouaojad}
Consider the rational function 
$$\rho(\xi_n)=\sum_{l=1}^{j+1} \tau_l(\xi_n-h)^{-l},$$
where $\mathrm{Im}(h)\neq 0$. Then 
$$\lim_{z_n\to 0^+}\mathcal{F}^{-1}_{\xi_n\to z_n} \rho=\begin{cases}
i\tau_1, \; &\mathrm{Im}(h)>0,\\
0, \; &\mathrm{Im}(h)<0.
\end{cases}$$
\end{lem}

In particular, if we have simplified the expression for $p_{\alpha-j}^\pm$ to the form 
\begin{equation}
\label{lknlknljknad}
p_{\alpha-j}^\pm(x,\xi)=\sum_{l=1}^{j+1} \tau_{jl}^\pm(x,\xi')(\xi_n-h_\pm(x,\xi'))^{\alpha-l},
\end{equation}
then 
\begin{equation}
\label{lakndalkdn}
s_{1-j}(x',\xi')=-i\tau_{j1}^+(x,\xi').
\end{equation}
We note that we can write $p_{\alpha-j}$ in the form \eqref{lknlknljknad} using a partial fraction decomposition. Computationally, it is of course hard to first of all find $p_{\alpha-j}^\pm$ in general and secondly to rewrite it in the form \eqref{lknlknljknad}. However, we can circumvent the form \eqref{lknlknljknad} using the following identity.

\begin{lem}
\label{laknlnadcllff}
For $a>0$ and $m\geq k\geq 0$ and $m\geq 1$
$$-\lim_{z_n\to 0^+}\mathcal{F}^{-1}_{\xi_n\to z_n} \left[\xi_n^{k}(\xi_n-ia)^{-m}\right]=
i^{k-m}\begin{pmatrix} k\\ 
m-1\end{pmatrix} a^{-m+k+1}.
$$
\end{lem} 

We can apply Lemma \ref{laknlnadcllff} to $\check{\mathfrak{q}}^{+}_j$ appearing in the factorization $\check{p}^+_{\alpha-j}:=\check{\mathfrak{q}}^{+}_jp^+_{\alpha}$, where $\check{p}^+_{\alpha-j}$ contains the leading $x_n$-derivatives on $a_0$ as described in Lemma \ref{complapdadadida}. We can by Lemma \ref{complapdadadida} write 
$$\check{\mathfrak{q}}^{+}_j(x,\xi',\xi_n)=\partial_{x_n}^{j-2}a_0(x)\sum_{k+l-m=-j} c_{k,l,m}(\alpha)\xi_n^k|\xi'|^l(\xi_n-i|\xi'|)^{-m}.$$
Applying Lemma \ref{laknlnadcllff} to this identity gives us that 
\begin{equation}
\label{lknlknlakndalkdn}
\lim_{z_n\to 0^+}\mathcal{F}^{-1}_{\xi_n\to z_n} \check{\mathfrak{q}}^+_{j}(x',0,\xi',\xi_n)=\partial_{x_n}^{j-2}a_0(x)\sum_{k+l-m=-j} i^{k-m}\begin{pmatrix} k\\ 
m-1\end{pmatrix}c_{k,l,m}(\alpha)|\xi'|^{1-j}.
\end{equation}
Equation \eqref{lknlknlakndalkdn} and Lemma \ref{complapdadadida} prove the following theorem. 

\begin{thm}
\label{alknalsdnaljkn}
For $\alpha\in (0,1]$ and $j=2,3,4,\ldots$, the symbol $s_{1-j}$ is of the form 
$$s_{1-j}(x',\xi')=c_j(\alpha)\partial_{x_n}^{j-2}a_0(x')|\xi'|_x^{1-j}+\tilde{a}_{1-j}(x',\xi'),$$
where $c_j(\alpha)$ is a universal, rational polynomial in $\alpha$ of degree $j-1$ and the symbol $\tilde{a}_{1-j}$ is a finite sum of products between powers of $|\xi'|_x$ and universal polynomials evaluated on the data 
\begin{align*}
\bigg(\xi, |\xi|_x^2, D_x |\xi|_x^2, &\ldots, D_x^j |\xi|_x^2,\\
&a_1(x,\xi), D_x a_1(x,\xi), \ldots, D_x^{j-1}a_1(x,\xi), \\
&\qquad\qquad a_0(x), D_x a_0(x), \ldots, D_x^{j-3}a_0(x),D_{x'}^{j-2}a_0(x)\bigg),
\end{align*}
by contracting with the metric tensor and setting $x_n=0$.
\end{thm}

We note that $c_0(\alpha)=0$ and $c_1(\alpha)=0$ for order reasons. The computation \eqref{lknlknlakndalkdn} shows that in terms of the universal polynomials $c_{k,l,m}(\alpha)$ of Lemma \ref{complapdadadida}, 
\begin{equation}
\label{firstcjsum}
c_j(\alpha)=\sum_{k+l-m=-j} i^{k-m}\begin{pmatrix} k\\ 
m-1\end{pmatrix}c_{k,l,m}(\alpha).
\end{equation}

\section{The half-space with potential}
\label{kjnjnad}

It remains to prove that $c_j(\alpha)\neq 0$ for $j>1$. We do so by computing $c_j(\alpha)$ for a particular geometry and particular potentials, rather than using the semi-explicit expression \eqref{firstcjsum}. The particular geometry is the half-space $\R^n_+$ and the particular potentials are $V_\mu(x',x_n)=\mathrm{e}^{-\mu x_n}$ for $\mu>0$ small. We remark that this example is not compactly supported, as assumed in Theorem \ref{aljnkajdna}, but the rapid decay still justifies the symbol computations and leads to the same polynomial $c_j(\alpha)$.

\subsection{Computations for a general potential in half-space}
We start out by investigating the polynomial $c_j(\alpha)$ by explicit computations for the Schr\"{o}dinger operator $A=\Delta+V$ in the Euclidean half-space 
$$\Omega = \R^n_+:=\{(x',x_n)\in \R^{n-1}\times \R: \; x_n\geq 0\}.$$
The computations exemplify the program outlined in Section \ref{sec:setup} for this particular geometry. We will write $L:=2i\xi\cdot\nabla_x$ and denote by $\Delta$ the Laplacian in the $x$-direction.

We compute as above that 
\begin{align*}
b_{-2}(x,\xi,\lambda)=&(|\xi|^2-\lambda)^{-1}, \\ 
b_{-3}(x,\xi,\lambda)=&0,\\
b_{-4}(x,\xi,\lambda)=&-V(x)(|\xi|^2-\lambda)^{-2},\\
b_{-5}(x,\xi,\lambda)=&-L(V)(x)(|\xi|^2-\lambda)^{-3},\\
b_{-6}(x,\xi,\lambda)=&-L^2(V)(x)(|\xi|^2-\lambda)^{-4}+(V(x)^2+\Delta(V)(x))(|\xi|^2-\lambda)^{-3}.
\end{align*}
Proceeding by induction, we can conclude the relation 
$$b_{-2-j}=(-\Delta b_{-j}+Lb_{-1-j}-Vb_{-j})(|\xi|^2-\lambda)^{-1}.$$

For instance, we have that 
\begin{align*}
p_{2\alpha-1}(x,\xi)=&0,\\
p_{2\alpha-2}(x,\xi)=&\alpha V(x)|\xi|^{2\alpha-2},\\
p_{2\alpha-3}(x,\xi)=&-\frac{\alpha(\alpha-1)}{2}L(V)(x)|\xi|^{2\alpha-4},\\
p_{2\alpha-4}(x,\xi)=&\frac{\alpha(\alpha-1)(\alpha-2)}{6}L^2(V)(x)|\xi|^{2\alpha-6}+\frac{\alpha(\alpha-1)}{2}(V(x)^2+\Delta(V)(x))|\xi|^{2\alpha-4}.
\end{align*}

Let us compute the first terms in the Wiener-Hopf factorization at the boundary $\R^{n-1}\times \{0\}$ of $\R^n_+$ in the asymptotic expansion $p^\pm \sim \sum_j p^\pm_{\alpha-j}$. We have 
$$p^\pm_\alpha(x,\xi)=(\xi_n\mp i|\xi'|)^\alpha,$$
as well as 
$$p^\pm_{\alpha-1}(x,\xi)=0.$$ 
We determine $p^\pm_{\alpha-2}=p^\pm_{\alpha}\mathfrak{q}^\pm_2$ from decomposing 
\begin{align*}
\frac{p_{\alpha-2}}{p_\alpha}-\frac{1}{p_\alpha}\sum_{\substack{k+l+|\beta|=2\\k,l<2}}\frac{1}{\beta!}\partial_\xi^\beta p^-_{\alpha-k}D_x^\beta p^+_{\alpha-l}=\mathfrak{q}^+_2+\mathfrak{q}^-_2.
\end{align*}
We compute the left hand side as $\alpha V(x)|\xi|^{-2}$ given $p^\pm_{\alpha-1}=0$ and the $x$-independence of $p_\alpha^\pm$. We proceed by decomposing
\begin{align*}
\alpha\frac{V}{|\xi|^2}=\frac{\alpha V}{2i|\xi'|}\left(\frac{1}{\xi_n-i|\xi'|}-\frac{1}{\xi_n+i|\xi'|}\right).
\end{align*}
Therefore
$$p^\pm_{\alpha-2}(x,\xi)=\pm \frac{\alpha V}{2i|\xi'|}(\xi_n\mp i|\xi'|)^{\alpha-1}.$$ 

We can determine $p^\pm_{\alpha-3}=p^\pm_{\alpha}\mathfrak{q}^\pm_3$ from decomposing 
\begin{equation}
\label{adlkandljn}
\frac{p_{\alpha-3}}{p_\alpha}-\frac{1}{p_\alpha}\sum_{\substack{k+l+|\beta|=3\\k,l<3}}\frac{1}{\beta!}\partial_\xi^\beta p^-_{\alpha-k}D_x^\beta p^+_{\alpha-l}=\mathfrak{q}^+_3+\mathfrak{q}^-_3.
\end{equation}
Write $L'=\xi'\cdot \nabla_{x'}$. Given $p^\pm_{\alpha-1}=0$ and the $x$-independence of $p_\alpha^\pm$, the equation \eqref{adlkandljn} takes the form
\begin{equation}
\label{aldnadln}
-\frac{\alpha(\alpha-1)}{2}L(V)|\xi|^{-4}+\frac{\alpha^2}{2|\xi'|} (\partial_{x_n}V+i|\xi'|^{-1}L'(V))|\xi|^{-2}=\mathfrak{q}^+_3+\mathfrak{q}^-_3.
\end{equation}

%\begin{align*}
%-\frac{\alpha(\alpha-1)}{2}L(V)&|\xi|^{-4}+\frac{\alpha}{2i|\xi'||\xi|^2} \nabla_\xi(\xi_n+i|\xi'|)\cdot \nabla V=\\
%=&\frac{\alpha(\alpha-1)}{2}(\xi_n\partial_{x_n} V+L'V)\bigg( \frac{1}{4i|\xi'|^3(\xi_n-i|\xi'|)}-\frac{1}{4i|\xi'|^3(\xi_n+i|\xi'|)}-\\
%&\qquad\qquad\qquad\qquad\qquad-\frac{1}{4|\xi'|^2(\xi_n-i|\xi'|)^2}-\frac{1}{4|\xi'|^2(\xi_n+i|\xi'|)^2}\bigg)\\
%&+\frac{\alpha}{4|\xi'|^2}\frac{\partial_{x_n}V+iL'V|\xi'|^{-1}}{\xi_n+i|\xi'|}-\frac{\alpha}{4|\xi'|^2}\frac{\partial_{x_n}V+L'V|\xi'|^{-1}}{\xi_n-i|\xi'|}.
%\end{align*}
%Therefore,
%\begin{align*}
%p^\pm_{\alpha-3}(x,\xi)=&\frac{\alpha(\alpha-1)}{2}(\xi_n\partial_{x_n} V+L'V)\bigg( \pm\frac{1}{4i|\xi'|^3}(\xi_n\mp i|\xi'|)^{\alpha-1}-\frac{1}{4|\xi'|^2}(\xi_n\mp i|\xi'|)^{\alpha-2}\bigg)\\
%&\mp\frac{\alpha}{4|\xi'|^2}\left(\partial_{x_n}V+iL'V|\xi'|^{-1}\right)(\xi_n\mp i|\xi'|)^{\alpha-1}.
%\end{align*} 

If we now perform the Fourier transforms giving 
$$s_{1-j}(x',\xi')=-\lim_{z_n\to 0^+}\mathcal{F}^{-1}_{\xi_n\to z_n} \mathfrak{q}^+_{j}(x',0,\xi',\xi_n),$$
we deduce the first few terms
\begin{align}
\label{symbolcomad}
s_1(x',\xi')=&-\alpha |\xi'|,\\
\nonumber
s_0(x',\xi')=&0,\\
\nonumber
s_{-1}(x',\xi')=&-\frac{\alpha}{2} V(x') |\xi'|^{-1},\\
\nonumber
s_{-2}(x',\xi')=&-\frac{\alpha^2}{4|\xi'|^2}\partial_{x_n} V -\frac{i\alpha}{4|\xi'|^3}L'(V).
\end{align}
To obtain the term $s_{-2}$, we used that 
\begin{align*}
s_{-2}=-\lim_{z_n\to 0^+}\mathcal{F}^{-1}_{\xi_n\to z_n} \mathfrak{q}^+_{3}=-i\mathrm{Res}_{\xi_n=i|\xi'|} \mathfrak{q}^+_{3}=-i\mathrm{Res}_{\xi_n=i|\xi'|} (\mathfrak{q}^+_{3}+\mathfrak{q}^-_{3}),
\end{align*}
and then applied the identity \eqref{aldnadln}.

We see from here that the polynomials $c_2(\alpha)$ and $c_3(\alpha)$ from Theorem \ref{alknalsdnaljkn} take the form 
\begin{equation}
\label{c2c3}
c_2(\alpha)=-\frac{\alpha}{2}\quad\mbox{and}\quad c_3(\alpha)=-\frac{\alpha^2}{4}.
\end{equation}

\subsection{Computing $c_j(\alpha)$}
\label{aonaldna}

Next we compute the polynomial $c_j(\alpha)$, using a clever choice of potential $V$ in the half-space $\R^n_+$. The outcome of this computation is given by.

\begin{prop}
\label{proponcj}
For every $j\ge2$, and $0<\alpha\le1$ it holds that $-2^j c_j(\alpha)$ is the Taylor coefficient at $0$ of degree $j-1$ of the holomorphic function  $\left(\frac{1+z}{1-z}\right)^{\!\alpha}$. In particular, we have that the polynomials $c_j(\alpha)$, $j=2,3,\ldots$, fit into a generating function
\begin{equation}
\label{eq:generating-function}
 \sum_{j=2}^{\infty}c_j(\alpha)t^{j-1} =-\frac{1}{2} \left[ \left(\frac{2+t}{2-t}\right)^{\!\alpha}-1 \right].
\end{equation}
\end{prop}
Note that 
$$\frac{1+z}{1-z}=1+2\sum_{k=1}^\infty z^{k}.$$
In particular, Proposition \ref{proponcj} ensures that
$c_j(1)=-2^{-j+1}$, $j\geq2$. This identity can also be directly obtained from the method of proof of \cite[Proposition 1.1]{leeuhl89} for the classical Calder\'{o}n problem.
\begin{proof}[Proof of Proposition \ref{proponcj}]
Fix $r=|\xi'|>0$ and work on the half-line $x=x_n\ge0$. Consider the family of potentials
$$V_\mu(x)=e^{-\mu x}.$$
We set
\[ A_0=D_x^2+r^2,\quad \mbox{and} \quad A_\varepsilon=A_0+\varepsilon V_\mu, \; \varepsilon>0.\] 
We can impose the restriction that $0<\mu<r$ because $r>0$ is fixed. We introduce the notation
$$K_\mu (z)=
 \frac{\bigl((z+i \mu )^2+r^2\bigr)^\alpha-(z^2+r^2)^\alpha}
 {(z+i \mu )^2-z^2}.$$
Since
\[
 D_xV_\mu =V_\mu(D_x+i\mu ),
\]
the divided-difference formula gives
\[
 \left.\frac{\rd}{\rd\varepsilon}\right|_{\varepsilon=0}A_\varepsilon^\alpha
 =V_\mu K_\mu (D_x).
\]

The principal Wiener--Hopf factors at $\varepsilon=0$ are
\[
 p_0^-(z)=(z+i r)^\alpha,
 \qquad
 p_0^+(z)=(z-i r)^\alpha.
\]
Writing the first variations of the Wiener--Hopf factors as $V_\mu a_\mu^\pm$, linearization of $p^-\!\ast p^+$ gives
\[
 K_\mu (z)=a_\mu ^-(z)p_0^+(z)+p_0^-(z+i \mu )a_\mu ^+(z).
\]
Thus the sum of the plus and minus components is
\begin{equation}
\label{eq:Fdef}
 F_{r,\mu }(z)=
 \frac{K_\mu (z)}{(z+i(r+\mu ))^\alpha(z-i r)^\alpha}.
\end{equation}
We have that $ F_{r,\mu }(z)=O(z^{-2})$ as $|z|\to \infty$. The first variation of the Dirichlet-to-Neumann symbol $s$ at $\varepsilon =0$ is therefore
$$\dot s(r,\mu ):=  \left.\frac{\rd}{\rd\varepsilon}\right|_{\varepsilon=0} s(r,\mu )
 =-\frac{1}{2\pi}\int_{\R}F_{r,\mu }(z)\,\rd z.$$

Using principal branches, \eqref{eq:Fdef} simplifies to
\begin{equation}
\label{eq:Fsimple}
 F_{r,\mu }(z)=\frac{1}{i \mu (2z+i \mu )}
 \left[
 \left(\frac{z+i(\mu -r)}{z-i r}\right)^{\!\alpha}
 -
 \left(\frac{z+i r}{z+i(\mu +r)}\right)^{\!\alpha}
 \right].
\end{equation}
Set
\[
 A_\mu(z)=\left(\frac{z+i(\mu -r)}{z-i r}\right)^{\!\alpha},
 \qquad
 f_\mu (z)=\frac{A_\mu (z)}{i \mu (2z+i \mu )}.
\]
The principal branch of $A_\mu $ is holomorphic in the closed lower half-plane.  Indeed, if its inner ratio were $-\tau$ with $\tau>0$, then
\[
 z=i\frac{\tau r+r-\mu }{1+\tau},
\]
which lies in the upper half-plane because $\mu <r$.  The reflection $z\mapsto-z-i \mu $ in \eqref{eq:Fsimple} yields
$$\int_{\R}F_{r,\mu }(z)\,\rd z
 =\mathrm{pv}\!\int_{\R}f_\mu (z)\,\rd z
 +\mathrm{pv}\!\int_{\R-i \mu }f_\mu (z)\,\rd z.$$
The individual integrals are symmetric principal values at infinity, and since $f_\mu (z)=(2i \mu z)^{-1}+O(z^{-2})$ their sum is absolutely convergent.

The only pole in the strip is $z_0=-i \mu /2$, and
\[
 A_\mu (z_0)=\left(\frac{2r-\mu }{2r+\mu }\right)^{\!\alpha}.
\]
Closing the first contour in the lower half-plane and then moving the second contour to the real axis gives
\[
 \mathrm{pv}\!\int_{\R}f_\mu (z)\,\rd z
 =\frac{\pi}{2\mu }-\frac{\pi}{\mu }A_\mu (z_0),
 \qquad
 \mathrm{pv}\!\int_{\R-i \mu }f_\mu (z)\,\rd z
 =\mathrm{pv}\!\int_{\R}f_\mu (z)\,\rd z+\frac{\pi}{\mu }A_\mu (z_0).
\]
Hence
$$\int_{\R}F_{r,\mu }(z)\,\rd z
 =\frac{\pi}{\mu } \left[ 1-\left(\frac{2r-\mu }{2r+\mu }\right)^{\!\alpha} \right],$$
and therefore
\begin{equation}
\label{eq:boxed-flat}
 \dot s(r,\mu )
 =-\frac{1}{2\mu }
 \left[
 1-\left(\frac{2r-\mu }{2r+\mu }\right)^{\!\alpha}
 \right].
\end{equation}

We note that the identity \eqref{eq:boxed-flat} implies the homogeneity property 
$$\dot s(\rho r,\rho \mu )=\rho^{-1}\dot s(r,\mu ),\quad\mbox{for $\rho>0$.}$$
At the linearized level $\dot{s}$, in our flat model $\R^n_+$ with the potential depending only on $x_n$, no tangential or nonlinear terms occur. In particular, $\dot{s}$ expands asymptotically in $r$ into terms of the form $(\partial_x^mV_\mu )(0)r^{1-j}$. A term $(\partial_x^mV_\mu )(0)r^{1-j}$ scales as $\rho^{m+1-j}$, so degree $1-j$ can contain only $m=j-2$.  Since $(\partial_x^mV_\mu )(0)=(-\mu )^m$, we conclude that
\begin{equation}
\label{eq:jet-expansion}
 \dot s(r,\mu )
 =\sum_{j=2}^{\infty}
 c_j(\alpha)(-\mu )^{j-2}r^{1-j}.
\end{equation}
A priori, the expression \eqref{eq:jet-expansion} should be interpreted as an asymptotic expansion in $r$ but a fortiori it holds as analytic functions. By analyticity we can set $\mu =-rt$ in the expression \eqref{eq:boxed-flat} for $\dot s$ and comparing with the expansion in $r$ in \eqref{eq:jet-expansion} proves that the polynomials $c_j(\alpha)$, $j=2,3,\ldots$, fit into the generating function \eqref{eq:generating-function}. We conclude that $-2^j c_j(\alpha)$ is the Taylor coefficient at $0$ of degree $j-1$ of the holomorphic function  $\left(\frac{1+z}{1-z}\right)^{\!\alpha}$.
\end{proof}

It remains to conclude that $c_j(\alpha) \neq 0$:
\begin{cor}
\label{adaldknal}
We have that
\begin{equation}
\label{eq:recursion}
 c_2(\alpha)=-\frac{\alpha}{2},\qquad
 c_3(\alpha)=-\frac{\alpha^2}{4},
\end{equation}
and, for $j\ge3$, the polynomials $c_j(\alpha)$ satisfy the recursion relation
\begin{equation}
\label{eq:recursion2}
 4j\,c_{j+1}(\alpha) =4\alpha\,c_j(\alpha)+(j-2)c_{j-1}(\alpha).
\end{equation}
In particular, $c_j(\alpha)<0$ for every $\alpha>0$ and $j\ge2$.
\end{cor}
\begin{proof}
The values of $ c_2(\alpha)$ and $ c_3(\alpha)$ in Equation \eqref{eq:recursion} follow from the computation \eqref{c2c3}. To obtain the recursion relation, write
\[
 \Phi_\alpha(z)=\left(\frac{1+z}{1-z}\right)^{\!\alpha} =\sum_{n=0}^\infty a_n(\alpha)z^n.
\]
By Proposition \ref{proponcj} we have that $c_j=-2^{-j}a_{j-1}$. We see that $\Phi_\alpha$ satisfies the differential equation
\begin{equation}
\label{lkngalkna}
(1-z^2)\Phi_\alpha'(z)=2\alpha\Phi_\alpha(z).
\end{equation}
In terms of the Taylor series, \eqref{lkngalkna} is equivalent to the recursion relation
\[
 (n+1)a_{n+1}=2\alpha a_n+(n-1)a_{n-1}, \quad n\geq 1.
\]
which shows \eqref{eq:recursion2}. Combining this recursion relation with $c_2(\alpha), c_3(\alpha)<0$, we conclude that $c_j(\alpha)<0$ for every $\alpha>0$ and $j \geq 2$.
\end{proof}

\end{document}